\documentclass[11pt]{amsart}

\usepackage{graphicx}
\usepackage{mathrsfs}
\usepackage{xcolor}
\usepackage{amssymb}
\usepackage{amsmath}
\usepackage{mathtools}
\usepackage{booktabs}
\usepackage{enumitem}
\usepackage[a4paper,margin=1.4in]{geometry}
\usepackage[hyperindex,breaklinks,colorlinks=true,
  linkcolor=blue!45!black,citecolor=blue!45!black,
  urlcolor=blue!55!black]{hyperref}

\numberwithin{equation}{section}

\newtheorem{prop}{Proposition}[section]
\newtheorem{theo}[prop]{Theorem}
\newtheorem{lemm}[prop]{Lemma}

\newtheorem{rema}[prop]{Remark}

\theoremstyle{definition}
\newtheorem{defi}[prop]{Definition}

\newcommand{\cL}{\mathcal L}
\newcommand{\cM}{\mathcal M}
\newcommand{\cU}{\mathcal U}

\newcommand{\bR}{\mathbb{R}}
\newcommand{\bS}{\mathbb{S}}

\newcommand{\dd}{\,\mathrm d}
\newcommand{\tr}{\operatorname{tr}}
\newcommand{\Div}{\operatorname{div}}

\newcommand{\ind}{\operatorname{ind}}
\newcommand{\Sing}{\operatorname{Sing}}
\newcommand{\Span}{\operatorname{span}}
\newcommand{\Id}{\operatorname{Id}}
\newcommand{\Rea}{\operatorname{Re}}

\newcommand{\abs}[1]{\lvert#1\rvert}
\newcommand{\inner}[2]{\left\langle #1,#2\right\rangle}
\newcommand{\Dbar}{\overline{\mathbb D}}
\newcommand{\Wulff}{\mathcal W_{F}}
\newcommand{\WulffR}[2]{\mathcal W_{#1}(#2)}
\newcommand{\halfball}[1]{B^{+}_{#1}}

\title[Hopf-Nitsche-type theorem]
{Hopf-Nitsche-type theorem for \\anisotropic capillary disks in a half $3$-space}

\author{Chenghang Lu}
\address{(C.L.)Albert-Ludwigs-Universit\"at Freiburg,
Mathematisches Institut, Ernst-Zermelo-Str.~1,
D-79104 Freiburg, Germany}
\email{chenghang.lu@math.uni-freiburg.de}
\thanks{C.L. is supported by  DFG,  through
IRTG 3132 (Grant No.~545046569).}

\author{Chao Xia}
\address{(C.X.) School of Mathematical Sciences,
Xiamen University, 361005, Xiamen, P.~R.~China}
\email{chaoxia@xmu.edu.cn}
\thanks{C.X. is supported by NSFC (Grant No. 12271449, 12671069, 12526203, 12526102) and the Natural Science Foundation of Fujian Province of China (Grant No. 2024J011008).}

\subjclass[2020]{53A10, 53C42, 35J67}
\keywords{anisotropic mean curvature, capillary surface, Hopf's theorem, Wulff shape 
}

\begin{document}
\raggedbottom

\begin{abstract}
In this paper, we prove that any compact immersed disk in
a half $3$-space that has constant anisotropic mean curvature and
satisfies the anisotropic capillary boundary condition is a homothetic Wulff cap.
\end{abstract}

\maketitle

\section{Introduction}

Hopf's theorem is a rigidity result for constant
mean curvature (CMC) surfaces, stating that every immersed CMC $2$-sphere in $\bR^3$ is a round sphere \cite{Hopf1983}. Its classical proof uses
the holomorphic Hopf differential. On the other hand, Wente \cite{Wente1986} has constructed immersed CMC $2$-torus in $\bR^3$. In fact, there exist immersed closed CMC surfaces of any given genus in $\bR^3$, see \cite{Kapouleas1991}. 
Nitsche \cite{Nitsche1985} considered the corresponding boundary version of Hopf's theorem about free boundary CMC disks
in a Euclidean ball and proved that every such disk is totally umbilic, and hence is a planar disk or a spherical cap. The free boundary
condition and the Terquem--Joachimsthal theorem make the boundary a line
of curvature, allowing Hopf's method  to work. The same mechanism underlies capillary boundary case and also
the classical planar-support version, see \cite{RosSouam97}. There are also other immersed free boundary or capillary CMC surfaces of other topology, see for example \cite{Wente1995}. 

Let $F:\bS^n\to(0,\infty)$ be a uniformly
elliptic integrand, i.e.
\begin{equation}\label{eq:strict-ellipticity}
 A_F(\xi):=D^2 F(\xi)+F(\xi)Id>0, \quad \xi\in \bS^n,
\end{equation}
 and \[
 \Phi:\bS^n\to \bR^{n+1}:\quad \Phi(\xi)=DF(\xi)+F(\xi)\xi
\] be the Cahn--Hoffman map. The image $\mathcal W_F=\Phi(\bS^n)$ is called the Wulff shape.
 For an immersed hypersurface $X:\Sigma^n\to \bR^{n+1}$, its anisotropic area is given by 
$$\mathcal{A}_F(X)=\int_\Sigma F(\nu) dA$$ where $\nu$ is a unit normal of $\Sigma$. Denote the anisotropic normal by  $\nu_F=\Phi(\nu)$.
The first variation of the anisotropic area gives rise to the anisotropic mean curvature 
$$H_F=\tr(d\nu_F).$$
 Koiso and Palmer \cite{KoisoPalmer2010} proved anisotropic analogue of Hopf's theorem saying that every immersed constant anisotropic mean curvature (CAMC) topological
$2$-sphere in $\bR^3$ is, up to translation and homothety, the Wulff shape. In contrast with the isotropic case, the classical Hopf differential is generally not holomorphic.  Koiso and Palmer \cite{KoisoPalmer2010} found an alternative approach to this problem by proving via Bers’ local elliptic expansion that each isolated anisotropic umbilic has negative index so as to apply the Poincaré–Hopf theorem. This approach can be applied in a much more general framework, see for example G\'alvez and Mira \cite{GalvezMira2020}.

In this paper, we consider Hopf-Nitsche-type problem in anisotropic setting. Let $$\mathbb D^2=\{x\in \bR^2: |x|<1\}$$ and
\[
 \mathbb R^3_+:=\{x\in\mathbb R^3:\inner{x}{E_3}>0\},
 \qquad E_3=(0,0,1).
\]
Let $X:\overline{\mathbb D^2}\to\overline{\mathbb R^3_+}$ be an oriented proper immersed disk. Here proper immersion means that $X(\mathbb D)\subset\mathbb R^3_+$ and
       $X(\partial\mathbb D)\subset\partial\mathbb R^3_+$. 
  The anisotropic
capillary energy is
\[
 \mathcal E_F(X)
 =\int_{\mathbb D^2} F(\nu)\dd A+\omega_0\,\mathcal A_W(X),
\]
where $\mathcal A_W(X)$ is the signed wetting area in the support plane.
The critical surfaces are of constant anisotropic mean curvature in
the interior and satisfies the Young boundary condition or the anisotropic capillary boundary condition
\[
 \inner{\nu_F}{-E_3}=\omega_0
 \quad\text{on }X(\partial \mathbb D).
\]
Translations and homotheties of $\mathcal W_F$ satisfying this boundary
condition are cut by the support plane into Wulff caps, or Winterbottom shapes, which are the corresponding
energy-minimizing configuration with fixed volume constraint \cite{Winterbottom1967}.

Our main result is the following Hopf-Nitsche-type theorem.

\begin{theo}\label{thm:main}
 Suppose that $F\in C^{4,\alpha}(\bS^2), 0<\alpha<1$, is a
positive uniformly elliptic integrand.
Let $X:\overline{\mathbb D^2}\to\overline{\mathbb R^3_+}$ be an oriented
$C^{4,\alpha}$ proper immersed disk with constant anisotropic mean curvature and satisfying the anisotropic capillary boundary condition. Then $X$ is a Wulff cap, up to translation and homothety. 
\end{theo}


We remark that the \(C^{4,\alpha}\)-regularity assumptions on \(F\) and
\(X\) in Theorem~\ref{thm:main} are technical. Although anisotropic
curvature is already defined at the \(C^2\) level, our approach requires
additional derivatives to control the regularity of the linearized and
reflected coefficients and to obtain the necessary Hessian asymptotics, see also Remark \ref{rem:coefficient-regularity}.

Very recently, Bastos \cite{Bastos2026} developed a capillary version of the Gálvez-Mira \cite{GalvezMira2020} method for classes of surfaces governed by elliptic equations and admitting a transitive capillary family. However, in Bastos’ formulation the capillary disk is obtained as \(\Sigma\cap\overline{\bR^3_+}\), where \(\Sigma\) is an immersed open surface. Consequently, every boundary point of the disk is an interior point of an a priori extension satisfying the same elliptic equation. In our setting, one is given only a CAMC immersion of a closed disk in a half-space, and the anisotropic capillary boundary condition does not, in general, produce a CAMC extension across the supporting plane. Thus Bastos' result \cite{Bastos2026} cannot imply ours.

We also mention several other rigidity results about anisotropic capillary CAMC surfaces. Jia, Wang,  Zhang and the second author \cite{JiaWangXiaZhang2023} proved the Alexandrov-type theorem that any compact embedded
anisotropic capillary CAMC hypersurface in a half-space is a truncated
Wulff shape, generalizing the closed case of He-Li-Ma-Ge \cite{HLMG09}. Guo and the second author \cite{GuoXia2025} proved the
Barbosa-do Carmo-type theorem that any compact weakly stable immersed anisotropic capillary CAMC
hypersurface in a half-space is a truncated
Wulff shape, generalizing the closed case of Palmer \cite{Palmer07}.

\

\noindent\textbf{Idea of the proof.}
Our proof starts from the basic comparison idea of Koiso and Palmer
\cite{KoisoPalmer2010}. For each \(p\in\Dbar\), we compare \(X\) with
the translated and scaled Wulff surface \(\WulffR{r_0}{c_p}\), where
\[
 r_0=\frac{2}{H_F},
 \qquad
 c_p=X(p)-r_0\nu_F(p),
\]
which has the same point, oriented tangent plane, and anisotropic mean
curvature as \(X\) at \(p\). The deviation from this model is measured
by the trace-free anisotropic second fundamental form
\[
 \mathring h_F
 =g_F\left(
   \left(d\nu_F-\frac{H_F}{2}\Id\right)\cdot,\cdot
 \right).
\]
Away from its zero set, \(\mathring h_F\) has signature \((1,1)\) and
defines two null line fields \(\mathcal L_+\) and \(\mathcal L_-\). At an interior
anisotropic umbilic (A-umbilic) point, the difference \(w\) of the graph functions of \(X\) and its
tangent Wulff surface satisfies a uniformly elliptic linear equation \eqref{eq:interior-linear-difference} and
has vanishing second derivative at $0$. The Hartman--Wintner expansion then gives
\[
 \mathring h_F=c_0D^2P_m+o(\varrho^{m-2}),
 \qquad
 P_m=c\,\Rea(z^m),\quad m\geq3,
\]
and hence
\[
 \ind_p(\mathcal L_\pm)=-\frac{m-2}{2}<0;
\]
this is the interior Koiso--Palmer mechanism, formulated in
Proposition~\ref{prop:interior-index}.

The principal new input is the analysis at a genuine anisotropic
capillary boundary. The anisotropic capillary boundary condition yields a conormal
boundary condition \eqref{eq:linear-conormal} for \(w\), but the surface
cannot in general be reflected across the support plane preserving the CAMC condition. We therefore make
a boundary-preserving change of variables and reflect the linear
elliptic equation for \(w\), rather than the immersed surface itself,
to obtain the boundary expansion of
Lemma~\ref{lem:boundary-expansion}. The analytic reflection determined
by the coefficients of this equation need not coincide with the
geometric reflection determined by the \(g_F\)-conormal \(\eta_F\);
Lemma~\ref{lem:analytic-geometric-compatibility} proves that these two
reflections agree to the order required to preserve the index. Together
with the anisotropic Joachimsthal identity
\[
 \mathring h_F(\tau,\eta_F)=0
 \qquad\text{on }\partial\mathbb D,
\]
this allows \(\mathcal L_+\) on one copy of the disk to be glued to \(\mathcal L_-\) on
the oppositely oriented copy. This doubling-and-interchange step is
topologically related to Bastos' construction \cite{Bastos2026}. Proposition~\ref{prop:boundary-index} shows
that every boundary A-umbilic has doubled index
\(-\frac{m-2}{2}<0\). The resulting line field on the doubled disk
\(\widehat\Sigma\simeq\bS^2\) therefore has only negative-index
singularities, contradicting the Poincar\'e--Hopf theorem unless
\(\mathring h_F\equiv0\).

\noindent\textbf{Organization of the paper.}
In Section~2, we collect the anisotropic notation and defines the Wulff
comparison surfaces. In Section~3, we derive the interior and boundary graph
equations, proves uniform ellipticity, and establishes the local geometric
asymptotics. In Section~4, we establish the index theorem for the null line
fields at interior and boundary A-umbilics and apply it to the oriented
double of the parameter disk, and    prove Theorem~\ref{thm:main}.
In Appendices~\ref{app:interior-expansion} and
\ref{app:boundary-expansion} we prove the interior and
boundary expansions, respectively. In Appendix~\ref{app:graph}, we derive the
graph equations and verify the ellipticity.

\

\noindent\textbf{Disclosure on AI assistance.}
The research problem and the principal mathematical ideas were formulated by the authors. The authors used AI-assisted tools, principally ChatGPT 5.6 to assist with some PDE estimates and to check computations. The authors verified and completed all mathematical
arguments and take full responsibility for the content. 

\

\section{Preliminaries}

We use the same letter $F$ for its one-homogeneous extension to
$\bR^3$:
\[
 F(z)=\abs z\,F(z/\abs z)\quad(z\neq0),
 \qquad F(0)=0.
\]
We write $DF$ and $D^2F$ for the gradient and Hessian on $\bS^2$, and
$\bar\nabla F$ and $\bar\nabla^2F$ for the Euclidean gradient
and Hessian of the extension. Set
\begin{equation*}
 A_F(\xi):=D^2F(\xi)+F(\xi)\Id.
\end{equation*}
The Cahn--Hoffman map is
\begin{equation*}
 \Phi(\xi):=\bar\nabla F(\xi)
 =DF(\xi)+F(\xi)\xi.
\end{equation*}
Strict ellipticity \eqref{eq:strict-ellipticity} means that
\[
 d\Phi_\xi=A_F(\xi):T_\xi\bS^2\longrightarrow T_\xi\bS^2
\]
is positive definite. Hence $\Phi$ is a diffeomorphism onto the smooth,
strictly convex closed surface
\(
\Wulff:=\Phi(\bS^2).
\)
Let
\[
 F^\circ(x):=\sup_{z\in\bS^2}\frac{\inner{x}{z}}{F(z)}.
\]
For $r_0>0$ and $x_0\in\bR^3$, we use the notation
\begin{equation*}
 \WulffR{r_0}{x_0}
 :=\{x\in\bR^3:F^\circ(x-x_0)=r_0\}
 =x_0+r_0\Wulff.
\end{equation*}

Let $X:M^2\to\bR^3$ be an oriented immersion. Throughout, $\nu:M^2\to
\bS^2$ denotes its chosen Euclidean unit normal. The anisotropic normal is
\begin{equation*}
 \nu_F:=\Phi\circ\nu=\Phi(\nu),\qquad
 d\nu_F=A_F(\nu)\circ d\nu.
\end{equation*}
The eigenvalues $\kappa_1^F,\kappa_2^F$ of $d\nu_F$ are real and are
called the anisotropic principal curvatures. We define the anisotropic
mean curvature by
\begin{equation}\label{eq:HF-def}
 H_F:=\tr(d\nu_F)=\kappa_1^F+\kappa_2^F.
\end{equation}
Although $d\nu_F$ need not be self-adjoint with respect to the usual first
fundamental form, it is self-adjoint with respect to the relative metric
\begin{equation*}
 g_F(u,v):=\inner{A_F(\nu)^{-1}u}{v}.
\end{equation*}
Indeed,
\[
 g_F(d\nu_F(u),v)=\inner{d\nu(u)}{v}
 =\inner{u}{d\nu(v)}=g_F(u,d\nu_F(v)).
\]
Define the symmetric trace-free anisotropic second fundamental form by
\begin{equation*}
 \mathring h_F(u,v)
 :=g_F\left(\left(d\nu_F-\frac{H_F}{2}\Id\right)u,v\right)
 =h(u,v)-\frac{H_F}{2}g_F(u,v).
\end{equation*}
At a point where $\mathring h_F\neq0$, this bilinear form is
nondegenerate and has
signature $(1,1)$. It therefore determines two unoriented null line
fields:
\[
 \cL_1(p)\cup\cL_2(p)
 =\{[v]\in\mathbb P(T_p\Dbar):(\mathring h_F)_p(v,v)=0\}.
\]
Here $\mathbb P(T_p\Dbar)$ is the set of unoriented lines in the tangent
space:
\[
 \mathbb P(T_p\Dbar)
 :=(T_p\Dbar\setminus\{0\})/\mathbb R^*,
 \qquad [v]=\Span\{v\}=\Span\{-v\}.
\]
Thus its elements are unoriented tangent lines. If
$p\in\partial\mathbb D$, then $T_p\Dbar$ is the full two-dimensional
tangent space of the manifold with boundary, not the inward tangent cone.
The common singular set of the two line fields is
\[
 \cU:=\{p:(\mathring h_F)_p=0\},
\]
whose points are called anisotropic umbilics, or A-umbilics.

\begin{lemm}
\label{lem:global-null-label}
On the oriented surface $\Dbar\setminus\cU$, the two null directions of
$\mathring h_F$ define two global continuous line fields $\cL_+$ and
$\cL_-$.
\end{lemm}

\begin{proof}
At every point outside $\cU$, the trace-free
$g_F$-self-adjoint operator
$d\nu_F-H_F\Id/2$ has eigenvalues $\lambda$ and $-\lambda$, with
$\lambda>0$. Its positive eigenspace therefore defines a global
unoriented line field $E_+$. Let $J_F$ denote rotation through $\pi/2$
with respect to the metric $g_F$ and the chosen orientation of the disk,
and let $R_\theta^F$
denote the corresponding rotation by $\theta$. The two lines
\[
 \cL_+:=R^{F}_{\pi/4}E_+,\qquad
 \cL_-:=R^{F}_{-\pi/4}E_+
\]
are the null lines of $\mathring h_F$. Since $E_+$ is an unoriented line
field, both rotations define global line fields.
\end{proof}

Set $r_0=2/H_F$. For every $p\in\Dbar$, define
\[
 c_p:=X(p)-r_0\nu_F(p).
\]
The translated and scaled Wulff shape
$\WulffR{r_0}{c_p}=c_p+r_0\Wulff$ is used as the comparison shape at
$p$. Near $X(p)$, it is parametrized by
\begin{equation}\label{eq:tangent-wulff-parametrization}
 \widetilde X_p(\xi):=c_p+r_0\Phi(\xi),
 \qquad \xi\in\bS^2\ \text{near }\nu(p).
\end{equation}
We give it the orientation for which its Euclidean unit normal field is
$\widetilde\nu(\xi):=\xi$. At $\xi=\nu(p)$,
\begin{equation}\label{eq:tangent-wulff-contact}
 \widetilde X_p(\nu(p))
 =c_p+r_0\Phi(\nu(p))=X(p),\qquad
 \widetilde\nu(\nu(p))=\nu(p).
\end{equation}
Thus the two surfaces pass through the same point and have the same
tangent plane and oriented normal there.

To compute the curvature of the comparison surface, note that
$d\widetilde X_p=r_0A_F(\xi)$ on $T_\xi\bS^2$, so its tangent plane at
$\widetilde X_p(\xi)$ is $\xi^\perp$. Its anisotropic normal
is $\widetilde\nu_F=\Phi(\xi)$, and hence
\begin{equation*}
 d\widetilde\nu_F=r_0^{-1}d\widetilde X_p,
 \qquad
 \widetilde H_F=\frac{2}{r_0}=H_F.
\end{equation*}
Consequently, the comparison Wulff shape has the same constant
anisotropic mean curvature as $X$. Moreover, $p$ is an A-umbilic if and
only if the immersed surface and $\WulffR{r_0}{c_p}$ have the same
second fundamental form at the contact point. Indeed, at an A-umbilic,
$d\nu_F(p)=r_0^{-1}\Id$; since the two normals agree and $A_F(\nu(p))$
is invertible, this is equivalent to equality of their ordinary shape
operators.
If $p\in\partial\mathbb D$, then
\begin{equation*}
 \inner{c_p}{E_3}=r_0\omega_0.
\end{equation*}
Every point on the comparison Wulff shape can be written as
\[
 y=c_p+r_0\Phi(\xi),\qquad \xi\in\bS^2,
\]
and its anisotropic normal at $y$ is $\Phi(\xi)$. The identity
\begin{equation*}
 \inner{y}{E_3}
 =r_0\bigl(\omega_0-\inner{\Phi(\xi)}{-E_3}\bigr)
\end{equation*}
therefore gives
\[
 y\in\partial\mathbb R^3_+
 \quad\Longleftrightarrow\quad
 \inner{\Phi(\xi)}{-E_3}=\omega_0.
\]
\section{Graph equations and local asymptotics}

\begin{lemm}\label{lem:interior-graph-difference}
Fix $p\in\mathbb D$ and let $\widetilde X$ be a
local parametrization of the comparison Wulff shape
$\WulffR{r_0}{c_p}$ defined in \eqref{eq:tangent-wulff-parametrization}.
By \eqref{eq:tangent-wulff-contact},
\[
 P:=T_{X(p)}X=T_{X(p)}\widetilde X=\nu(p)^\perp.
\]
Choose a constant unit vector $\zeta$ with
$\inner{\zeta}{\nu(p)}\neq0$. In a neighborhood of $p$, the two
surfaces can be written as graphs $u$ and $\widetilde u$ over $P$ in
the direction $\zeta$. There is a \(C^{4,\alpha}\) function $G_\zeta$ whose Hessian
is uniformly positive definite on every bounded subset of $\mathbb R^2$,
such that
\begin{equation}\label{eq:interior-graph-equations}
 \Div DG_\zeta(Du)=-\inner{\zeta}{\nu(p)}H_F,
 \qquad
 \Div DG_\zeta(D\widetilde u)=-\inner{\zeta}{\nu(p)}H_F.
\end{equation}
Consequently, for $w=u-\widetilde u$,
\begin{equation}\label{eq:interior-linear-difference}
 \partial_i(a^{ij}w_j)=0,
 \qquad
 a^{ij}:=\int_0^1(G_\zeta)_{q_iq_j}
 \bigl(D\widetilde u+\theta Dw\bigr)\dd\theta.
\end{equation}
If $p$ is an A-umbilic, then
\begin{equation}\label{eq:interior-second-order-contact}
 w(0)=0,\qquad Dw(0)=0,\qquad D^2w(0)=0.
\end{equation}
\end{lemm}

\begin{proof}
Let $\pi_\zeta:\mathbb R^3\to P$ be the projection along the lines
parallel to $\zeta$. Since $\zeta\notin P$, the restriction of
$d\pi_\zeta$ to the common tangent plane $P$ is the identity. The
inverse function theorem therefore gives the two local graph
representations in the statement.

Choose an oriented orthonormal basis $(e_1,e_2)$ of $P$ with
$e_1\times e_2=\nu(p)$, translate $X(p)$ to the origin, and identify
$P$ with $\mathbb R^2$ through this basis. The graph representations are
\[
 X(x)=x_1e_1+x_2e_2+u(x)\zeta,
 \qquad
 \widetilde X(x)=x_1e_1+x_2e_2+\widetilde u(x)\zeta.
\]
For $q=(q_1,q_2)$, define
\[
 V_\zeta(q)
 :=\nu(p)+q_1(\zeta\times e_2)+q_2(e_1\times\zeta),
 \qquad
 G_\zeta(q):=F(V_\zeta(q)).
\]
The first variation of the graph energy and the ellipticity calculation
are given in Appendix~\ref{app:graph}. They yield
\eqref{eq:interior-graph-equations}; subtraction and the fundamental
theorem of calculus then give \eqref{eq:interior-linear-difference}.

Since both graphs are taken over their common tangent plane,
$u(0)=\widetilde u(0)=0$ and $Du(0)=D\widetilde u(0)=0$. If $p$ is an
A-umbilic, the comparison
property in Section~2 shows that the surface and
the model have the same ordinary second fundamental form at $p$. For
either graph, differentiating $\inner{\nu_u}{X_j}=0$ gives
\[
 \inner{\partial_i\nu_u}{X_j}
 =-\inner{\nu_u}{\zeta}u_{ij}.
\]
Since $\inner{\nu(p)}{\zeta}\neq0$, equality of the second fundamental
forms implies that the two Hessians agree at the origin. This proves
\eqref{eq:interior-second-order-contact}.
\end{proof}

\begin{rema}\label{rem:coefficient-regularity}
The assumption $F,X\in C^{4,\alpha}$ gives
$A=(a^{ij})\in C^{2,\alpha}$. In the boundary coordinates introduced
below, the coefficients are $C^{1,\alpha}$ up to $t=0$, and the reflected
coefficient matrix defined in Appendix~\ref{app:boundary-expansion} is
Lipschitz across $t=0$. This regularity is sufficient for the boundary
Hartman--Wintner expansion and the Schauder estimate used to obtain the
second-order asymptotic.
\end{rema}

Fix $p\in\partial\mathbb D$. Let $\tau$ be a unit tangent to the boundary and let
$\mu$ be the unit outward conormal of $\partial\mathbb D$ in the surface.
Set $\eta:=-\mu$, the inward conormal, and choose the sign of $\tau$ so
that
\[
 \tau\times\eta=\nu(p).
\]
Because $\inner{X}{E_3}>0$ in the interior and the surface meets the
support plane transversely,
\[
 b:=\inner{\eta}{E_3}=\inner{\mu}{-E_3}>0.
\]
The decomposition of $-E_3$ in $\Span\{\nu(p),\eta\}$ is
\begin{equation*}
 -E_3=a\nu(p)-b\eta,
 \qquad a^2+b^2=1.
\end{equation*}
Define the graph direction
\begin{equation*}
 \bar\nu:=b\nu(p)+a\eta.
\end{equation*}
Then
\begin{equation}\label{eq:oblique-direction-identities}
 \abs{\bar\nu}=1,\qquad
 \bar\nu\in E_3^\perp=T_{X(p)}\partial\mathbb R^3_+,\qquad
 \inner{\bar\nu}{\nu(p)}=b>0,\qquad
 \tau\times\bar\nu=-E_3.
\end{equation}
In particular, $\bar\nu$ is transverse to the common tangent plane
\(
T_{X(p)}X=\Span\{\tau,\eta\}
\).

We apply the graph construction of Lemma~\ref{lem:interior-graph-difference}
with
\[
 (e_1,e_2)=(\tau,\eta),\qquad \zeta=\bar\nu,
\]
using a half-disk as the base domain. Translate $X(p)$ to the origin.
The surface can be written as a graph in the direction $\bar\nu$:
\begin{equation}\label{eq:adapted-graph}
 X(s,t)=s\tau+t\eta+u(s,t)\bar\nu,
 \qquad t\geq0.
\end{equation}
Indeed,
\begin{equation*}
 \inner{X(s,t)}{E_3}=bt,
\end{equation*}
because both $\tau$ and $\bar\nu$ lie in the support plane. Thus the
graph points with $t=0$ lie in $\partial\mathbb R^3_+$, whereas those
with $t>0$ lie in $\mathbb R^3_+$. The comparison Wulff shape
$\WulffR{r_0}{c_p}$ is a graph
over the same tangent plane and in the same graph direction:
\[
 \widetilde X(s,t)=s\tau+t\eta+\widetilde u(s,t)\bar\nu,
 \qquad t\geq0.
\]
For $q=(q_1,q_2)\in\mathbb R^2$, set
\[
 V(q):=\nu(p)-bq_1\tau-q_2E_3,
 \qquad G(q):=F(V(q)).
\]
Differentiating \eqref{eq:adapted-graph} and using
\eqref{eq:oblique-direction-identities} gives
\begin{equation}\label{eq:oblique-graph-normal}
 X_s\times X_t
 =(\tau+u_s\bar\nu)\times(\eta+u_t\bar\nu)
 =\nu(p)-bu_s\tau-u_tE_3=V(Du).
\end{equation}
Thus the oriented unit normal is $\nu_u:=V(Du)/\abs{V(Du)}$. Since
$\bar\nabla F$ is homogeneous of degree zero,
\begin{equation*}
 G_{q_1}(Du)=-b\inner{\nu_F}{\tau},
 \qquad
 G_{q_2}(Du)=\inner{\nu_F}{-E_3}.
\end{equation*}
The first-variation calculation in Appendix~\ref{app:graph} yields
\begin{align*}
 \partial_iG_{q_i}(Du)&=-bH_F &&\text{for }t>0,\\
 G_{q_2}(Du)&=\omega_0 &&\text{for }t=0.
\end{align*}
The model satisfies the same equation and boundary condition. Set
\begin{equation*}
 \rho(q):=\abs{V(q)},
 \qquad \nu(q):=\frac{V(q)}{\rho(q)}.
\end{equation*}
For $\zeta=(\zeta_1,\zeta_2)\in\bR^2$, set
$\xi=-b\zeta_1\tau-\zeta_2E_3$, and let $\xi^\top$ be its Euclidean
projection onto $\nu(q)^\perp$. The calculation in
Appendix~\ref{app:graph} gives
\begin{equation}\label{eq:G-Hessian-exact}
 D^2G(q)[\zeta,\zeta]
 =\frac{1}{\rho(q)}
   \inner{A_F(\nu(q))\xi^\top}{\xi^\top}.
\end{equation}
In particular, for every $M<\infty$ there are constants
$0<\lambda_M\leq\Lambda_M<\infty$ such that
\begin{equation}\label{eq:G-uniform-ellipticity}
 \lambda_M\abs\zeta^2
 \leq D^2G(q)[\zeta,\zeta]
 \leq\Lambda_M\abs\zeta^2
 \qquad(\abs q\leq M).
\end{equation}

Define $w:=u-\widetilde u$ and, for $0\leq\theta\leq1$, define
\begin{equation*}
 q_\theta:=D\widetilde u+\theta(Du-D\widetilde u)
 =D\widetilde u+\theta Dw.
\end{equation*}

Applying the fundamental theorem of calculus to each component of \(DG\), we obtain 
\begin{equation*}
 G_{q_i}(Du)-G_{q_i}(D\widetilde u)
 =\int_0^1G_{q_iq_j}(q_\theta)w_j\dd\theta
 =a^{ij}w_j,
\end{equation*}
where
\begin{equation}\label{eq:aij-def}
 a^{ij}(s,t):=\int_0^1G_{q_iq_j}(q_\theta(s,t))\dd\theta.
\end{equation}
Subtracting the interior and boundary equations gives
\begin{align}
 \partial_i(a^{ij}w_j)&=0 &&\text{for }t>0,\notag\\
 a^{2j}w_j&=0 &&\text{for }t=0.
 \label{eq:linear-conormal}
\end{align}

On a relatively compact graph neighborhood, choose $M<\infty$ with
\[
 \abs{Du},\abs{D\widetilde u}\leq M.
\]
Then $\abs{q_\theta}\leq M$. Integrating
\eqref{eq:G-uniform-ellipticity} in $\theta$, we obtain for every
$\zeta\in\bR^2$
\begin{equation*}
 \lambda_M\abs\zeta^2
 \leq a^{ij}\zeta_i\zeta_j
 \leq\Lambda_M\abs\zeta^2.
\end{equation*}
Thus $A=(a^{ij})$ is symmetric, uniformly positive definite, and belongs
to $C^{2,\alpha}$.

Since both graphs are taken over their common tangent plane,
we have
\begin{equation*}
 w(0)=u(0)-\widetilde u(0)=0,\qquad
 Du(0)=D\widetilde u(0)=0,
 \qquad Dw(0)=0.
\end{equation*}

At an A-umbilic, $\mathring h_F(p)=0$. Hence
\[
 d\nu_F(p)=\frac{H_F}{2}\Id.
\]
The comparison Wulff surface also satisfies
$d\widetilde\nu_F=(H_F/2)\Id$ at the contact point. Since the normals agree and
$A_F(\nu(p))$ is invertible, their ordinary second fundamental forms
also agree.
For a graph $X_i=e_i+u_i\bar\nu$, where
$(e_1,e_2)=(\tau,\eta)$, the identity
$\inner{\nu_u}{X_j}=0$ gives
\begin{equation*}
 \inner{\partial_i\nu_u}{X_j}
 =-\inner{\nu_u}{\bar\nu}u_{ij}.
\end{equation*}
At the origin, $\nu_u=\nu_{\widetilde u}=\nu(p)$ and
$\inner{\nu(p)}{\bar\nu}=b>0$. Equality of the ordinary second fundamental
forms therefore gives, for every
$i,j\in\{1,2\}$,
\[
 -b u_{ij}(0)=-b\widetilde u_{ij}(0),
\]
and hence
\begin{equation*}
 D^2u(0)=D^2\widetilde u(0),
 \qquad D^2w(0)=0.
\end{equation*}
Combining this with the first-order contact conditions, we obtain
\begin{equation}\label{eq:boundary-second-order-contact}
 w(0)=0,\qquad Dw(0)=0,\qquad D^2w(0)=0.
\end{equation}

\begin{prop}\label{prop:joachimsthal}
For each unit tangent vector $\tau$ to $\partial\mathbb D$, choose
$\eta_F$ so that
\begin{equation*}
 g_F(\tau,\eta_F)=0,
 \qquad g_F(\eta_F,\eta_F)=1.
\end{equation*}
Under the boundary hypotheses of Theorem~\ref{thm:main}, $\tau$ is an
eigenvector of $d\nu_F$. Consequently,
\begin{equation*}
 \mathring h_F(\tau,\eta_F)=0
 \quad\text{on }\partial\mathbb D.
\end{equation*}
\end{prop}

\begin{proof}
For completeness, we give a short proof; see also
\cite[Proposition~4.1]{GuoXia2025}. Differentiating the Young condition
along $\tau$ gives
\[
 0=\inner{d\nu_F(\tau)}{-E_3}.
\]
Since $d\nu_F(\tau)$ is tangent to the surface and
$-E_3=a\nu-b\eta$, with $b>0$, it follows that
\[
 \inner{d\nu_F(\tau)}{\eta}=0.
\]
Hence $d\nu_F(\tau)$ is parallel to $\tau$, so $\tau$ is an
eigenvector of $d\nu_F$. Since $g_F(\tau,\eta_F)=0$, we obtain
\[
 \mathring h_F(\tau,\eta_F)
 =g_F\left(\left(d\nu_F-\frac{H_F}{2}\Id\right)\tau,\eta_F\right)
 =0.
\]
\end{proof}

\begin{lemm}\label{lem:interior-expansion}
Let $A=(a^{ij})\in C^{1,\alpha}(B_\rho)$ be symmetric and uniformly
positive definite, and suppose that $w\in C^{3,\alpha}(B_\rho)$ satisfies
\[
 \partial_i(a^{ij}w_j)=0,\qquad
 w(0)=0,\quad Dw(0)=0,\quad D^2w(0)=0.
\]
Then either $w$ vanishes in a neighborhood of the origin, or there exist
an integer $m\geq3$, a nonzero homogeneous harmonic polynomial $P_m$ of
degree $m$, and an orientation-preserving linear coordinate system in
which, with $\varrho=\abs q$,
\begin{align}
 w&=P_m+o(\varrho^m),\notag\\
 Dw&=DP_m+o(\varrho^{m-1}),\notag\\
 D^2w&=D^2P_m+o(\varrho^{m-2}),\label{eq:int-exp-2}
\end{align}
as $q\to0$.
\end{lemm}

The first two expansions follow from the Hartman--Wintner theorem
\cite[Theorem~(H.--W.) and Remark~1.1, p.~231]{Alessandrini1987}; the
Hessian expansion follows by rescaling and the interior Schauder estimate
\cite[Theorem~6.2]{GilbargTrudinger2001}. Details are given in
Appendix~\ref{app:interior-expansion}.

For the boundary version, write
\[
 \halfball{\rho}:=\{(s,t):s^2+t^2<\rho^2,\ t\geq0\}.
\]

\begin{lemm}\label{lem:boundary-expansion}
Let $A=(a^{ij})\in C^{2,\alpha}(\halfball\rho)$ be symmetric and uniformly
positive definite, and suppose that $w\in C^{3,\alpha}(\halfball\rho)$
satisfies
\begin{equation*}
 \partial_i(a^{ij}w_j)=0\quad(t>0),
 \qquad e_t\cdot A\nabla w=0\quad(t=0),
\end{equation*}
together with the three conditions in
\eqref{eq:boundary-second-order-contact}. Then exactly
one of the following alternatives holds:
\begin{enumerate}[label=\textup{(\alph*)}]
 \item $w$ vanishes in a neighborhood of the origin relative to the
 closed half-disk;
 \item there exist orientation-preserving $C^{2,\alpha}$ coordinates
 $z=x+iy$ that preserve the boundary line and upper half-plane, an integer
 $m\geq3$, and a constant $c\neq0$ such that, with $\varrho=\abs z$,
 \begin{align}
  w(x,y)&=P_m(x,y)+o(\varrho^m),\notag\\
  Dw(x,y)&=DP_m(x,y)+o(\varrho^{m-1}),\notag\\
  D^2w(x,y)&=D^2P_m(x,y)+o(\varrho^{m-2}),\label{eq:bd-exp-2}
 \end{align}
 where
 \begin{equation}\label{eq:even-harmonic-polynomial}
  P_m(x,y)=c\Rea(z^m),
 \end{equation}
 as $z\to0$ in the closed upper half-plane.
\end{enumerate}
\end{lemm}

The proof in Appendix~\ref{app:boundary-expansion} changes coordinates
so that the conormal condition becomes a homogeneous Neumann condition,
reflects the equation across the boundary line, and applies the
Hartman--Wintner theorem; see
\cite[p.~579]{AlessandriniMagnanini1992}. A boundary Schauder estimate
then gives the expansion of $D^2w$.

\begin{lemm}\label{lem:tensor-transfer}
Let $u$ and $\widetilde u$ be local graph functions for the surface and
its comparison Wulff shape over their common tangent plane, and write
$w=u-\widetilde u$. Suppose that there are local coordinates $q$,
centered at the contact point, in which
\[
 Dw=DP_m+o(\varrho^{m-1}),\qquad
 D^2w=D^2P_m+o(\varrho^{m-2}),\qquad m\geq3,
\]
where $\varrho=\abs q$. Then, in the same coordinates, there is a
constant $c_0\neq0$ such that
\begin{equation}\label{eq:tracefree-hF-leading-hessian}
 \mathring h_F=c_0D^2P_m+o(\varrho^{m-2}).
\end{equation}
\end{lemm}

\begin{proof}
We work in fixed graph coordinates. In the boundary case, let
$v=w\circ\Psi$, where $\Psi$ is the nonlinear
$C^{2,\alpha}$ change of variables used in
Lemma~\ref{lem:boundary-expansion}. The chain rule gives
\[
 D^2v
 =(D\Psi)^T(D^2w\circ\Psi)D\Psi
 +\sum_{k=1}^2(w_k\circ\Psi)D^2\Psi_k.
\]
The last sum is $O(\varrho^{m-1})=o(\varrho^{m-2})$ and therefore does
not affect the leading term.

Let $\zeta$ be the constant graph direction in these coordinates.
It is transverse to the tangent plane, so
$\inner{\zeta}{\nu(0)}\neq0$. Since
$Du(0)=D\widetilde u(0)$, one has
$\nu_u(0)=\nu_{\widetilde u}(0)=\nu(0)$.
The exact graph formula for the two second fundamental forms is
\begin{equation}\label{eq:II-graph-difference}
 h_u-h_{\widetilde u}
 =-\inner{\zeta}{\nu_u}D^2u
  +\inner{\zeta}{\nu_{\widetilde u}}D^2\widetilde u.
\end{equation}
Since the coefficient $\inner{\zeta}{\nu_u}$ depends smoothly on $Du$,
\[
 \inner{\zeta}{\nu_u}-\inner{\zeta}{\nu_{\widetilde u}}
 =O(\abs{Dw}),\qquad
 \inner{\zeta}{\nu_u}-\inner{\zeta}{\nu(0)}=O(\varrho),
\]
where the second estimate follows from the $C^2$ regularity of the graph
and $Du(0)=D\widetilde u(0)$. Substituting
$D^2u=D^2\widetilde u+D^2w$ into
\eqref{eq:II-graph-difference} therefore gives
\begin{equation}\label{eq:II-graph-leading-difference}
 h_u-h_{\widetilde u}
 =-\inner{\zeta}{\nu(0)}D^2w
  +O\bigl(\abs{Dw}+\varrho\abs{D^2w}\bigr).
\end{equation}
For either graph,
\[
 \mathring h_F=h-\frac{H_F}{2}g_F,
\]
as follows directly by substituting $d\nu_F=A_Fd\nu$ into the definition
of $g_F$. The comparison Wulff shape has
$\mathring h_F\equiv0$. Expressing both forms in the fixed graph
coordinates gives
\[
 \mathring h_F=(h_u-h_{\widetilde u})
   -\frac{H_F}{2}(g_F^u-g_F^{\widetilde u}).
\]
The coefficients of $g_F$ depend smoothly on the first derivatives of
the graph function. Combining
this observation with \eqref{eq:II-graph-leading-difference} yields
\begin{equation*}
 \mathring h_F=-\inner{\zeta}{\nu(0)}D^2w
 +O\bigl(\abs{Dw}+\varrho\abs{D^2w}\bigr).
\end{equation*}
Since $Dw=O(\varrho^{m-1})$ and $D^2w=O(\varrho^{m-2})$, the remainder
term in this formula satisfies
\[
 O\bigl(\abs{Dw}+\varrho\abs{D^2w}\bigr)
 =O(\varrho^{m-1})=o(\varrho^{m-2}).
\]
Substituting $D^2w=D^2P_m+o(\varrho^{m-2})$ proves
\eqref{eq:tracefree-hF-leading-hessian} with
$c_0=-\inner{\zeta}{\nu(0)}\neq0$.
\end{proof}

\begin{lemm}\label{lem:analytic-geometric-compatibility}
Let $p$ be a boundary A-umbilic and use the graph parametrization
\eqref{eq:adapted-graph}. Let $A=(a^{ij})$ be defined by
\eqref{eq:aij-def}. Assume that
alternative~\textup{(b)} of Lemma~\ref{lem:boundary-expansion} holds, and let
$m\geq3$ be the degree of the polynomial
$P_m=c\Rea(z^m)$ in \eqref{eq:even-harmonic-polynomial}.

On $t=0$, define
\begin{equation*}
 \beta_A(s):=\frac{a^{12}(s,0)}{a^{22}(s,0)},\qquad
 \beta_F(s):=
 \frac{G_{q_1q_2}(Du(s,0))}{G_{q_2q_2}(Du(s,0))}.
\end{equation*}
Then
\begin{equation}\label{eq:two-boundary-directions-close}
 \beta_A(s)-\beta_F(s)=O(\abs{s}^{m-1}),\qquad
 \partial_s(\beta_A-\beta_F)(s)=O(\abs{s}^{m-2}).
\end{equation}
\end{lemm}

\begin{proof}
On $t=0$, set
\[
 N_A:=\partial_t+\beta_A(s)\partial_s,
 \qquad
 N_F:=\partial_t+\beta_F(s)\partial_s.
\]
Since $a^{21}=a^{12}$, the boundary condition
\eqref{eq:linear-conormal} becomes
\[
 a^{2j}w_j=a^{22}N_Aw=0.
\]

We first show that $dX(N_F)$ is $g_F$-orthogonal to the boundary
tangent. Write $X_i=e_i+u_i\bar\nu$, where
$(e_1,e_2)=(\tau,\eta)$, and let
$M_u=(g_F(X_i,X_j))$. Let $Z_1,Z_2$ be the Euclidean dual basis,
so that $\inner{Z_i}{X_j}=\delta_{ij}$. By the definition of $g_F$,
\[
 \bigl(\inner{A_F(\nu)Z_i}{Z_j}\bigr)_{ij}=M_u^{-1}.
\]
The tangential projections of $\partial_{q_1}V=-b\tau$ and
$\partial_{q_2}V=-E_3$ are $-bZ_1$ and $-bZ_2$, respectively, since
\[
 \inner{\partial_{q_i}V}{X_j}=-b\delta_{ij}.
\]
Thus \eqref{eq:G-Hessian-exact} gives
\[
 D^2G(Du)=\frac{b^2}{\abs{V(Du)}}M_u^{-1}.
\]
Writing $M_u=(m_{ij})$, we obtain
\[
 \beta_F
 =\frac{G_{q_1q_2}(Du)}{G_{q_2q_2}(Du)}
 =-\frac{m_{12}}{m_{11}},
 \qquad
 g_F(X_1,X_2+\beta_FX_1)=0.
\]
Hence $dX(N_F)=X_2+\beta_FX_1$ spans the same line as
$\eta_F$.

Along $t=0$, we have
\[
 q_\theta=D\widetilde u+\theta Dw
 =Du-(1-\theta)Dw,
 \qquad
 A=\int_0^1D^2G(q_\theta)\dd\theta.
\]
The coordinate changes in Lemma~\ref{lem:boundary-expansion}
and their inverses have bounded first and second derivatives near
the origin. The expansions in that lemma therefore give
\[
 \abs{Dw(s,0)}\leq C\abs{s}^{m-1},
 \qquad
 \abs{D^2w(s,0)}\leq C\abs{s}^{m-2}
\]
in the $(s,t)$-coordinates of \eqref{eq:adapted-graph}.
By the fundamental theorem of calculus,
\[
 a^{ij}-G_{q_iq_j}(Du)
 =-\int_0^1\int_0^{1-\theta}
 G_{q_iq_jq_k}(Du-\sigma Dw)w_k\dd\sigma\dd\theta.
\]
Since $D^2G\in C^{2,\alpha}$, this identity and its tangential
derivative yield
\begin{align}
 \abs{A-D^2G(Du)}
 &\leq C\abs{Dw}=O(\abs{s}^{m-1}),
 \label{eq:A-D2G-zero-order}\\
 \abs{\partial_s\bigl(A-D^2G(Du)\bigr)}
 &\leq C\bigl(\abs{Dw}+\abs{D^2w}\bigr)
 =O(\abs{s}^{m-2}).
 \label{eq:A-D2G-first-order}
\end{align}
On a sufficiently small boundary interval, uniform ellipticity gives
\[
 a^{22}\geq\lambda,\qquad
 G_{q_2q_2}(Du)\geq\lambda
 \qquad(\lambda>0).
\]
The matrix entries and their first tangential derivatives are bounded
on this interval. Applying the quotient rule to
\eqref{eq:A-D2G-zero-order}--\eqref{eq:A-D2G-first-order}
proves \eqref{eq:two-boundary-directions-close}. In particular,
$\beta_A(0)=\beta_F(0)$.

We now derive the reflection estimate used in the boundary index
calculation. For $\abs x$ and $\abs y$ small, define
\[
 \Psi_A^0(x,y)=(x+\beta_A(x)y,y),
 \qquad
 \Psi_F^0(x,y)=(x+\beta_F(x)y,y).
\]
Along $y=0$,
\[
 D\Psi_A^0(\partial_y)=N_A,
 \qquad
 D\Psi_F^0(\partial_y)=N_F.
\]
In the original $(s,t)$-coordinates, the second map defines the
reflection
\[
 (\sigma+\beta_F(\sigma)t,t)
 \longmapsto
 (\sigma-\beta_F(\sigma)t,-t).
\]

In the $\Psi_A^0$-coordinates, the coefficient matrix is diagonal and
positive definite at the origin. Let $L$ be the fixed positive
diagonal map used in the proof of
Lemma~\ref{lem:boundary-expansion} to make this matrix a positive
multiple of the identity, and set
\[
 \Psi_A:=\Psi_A^0\circ L,
 \qquad
 \Psi_F:=\Psi_F^0\circ L.
\]
Since $L$ commutes with $R=\operatorname{diag}(1,-1)$, the reflection
for the linear boundary condition is $R$ in the normalized
$\Psi_A$-coordinates. In these coordinates, the reflection defined
using $\Psi_F$ is
\[
 \Theta:=\Psi_A^{-1}\circ\Psi_F,
 \qquad
 \mathscr R_F:=\Theta\circ R\circ\Theta^{-1}.
\]

Set $\delta=\beta_A-\beta_F$. Direct differentiation gives
\[
 \Psi_A^0-\Psi_F^0=(\delta(x)y,0),
 \qquad
 D\Psi_A^0-D\Psi_F^0
 =\begin{pmatrix}
   y\delta'(x)&\delta(x)\\
   0&0
  \end{pmatrix}.
\]
Fix $0<c_1<c_2$. On $c_1r\leq\abs q\leq c_2r$, the estimates in
\eqref{eq:two-boundary-directions-close} give
\[
 \Psi_A(q)-\Psi_F(q)=O(r^m),
 \qquad
 D\Psi_A(q)-D\Psi_F(q)=O(r^{m-1}).
\]
Both maps are local $C^2$ diffeomorphisms: indeed,
\[
 \det D\Psi_A^0=1+\beta_A'(x)y,
 \qquad
 \det D\Psi_F^0=1+\beta_F'(x)y,
\]
and these determinants are nonzero near the origin. Since $\Psi_A$
and its inverse have bounded first and second derivatives there, the
preceding estimates imply
\[
 \begin{aligned}
  \Theta(q)&=q+O(r^m),
  &D\Theta(q)&=I+O(r^{m-1}),\\
  \Theta^{-1}(q)&=q+O(r^m),
  &D\Theta^{-1}(q)&=I+O(r^{m-1}).
 \end{aligned}
\]
Consequently,
\begin{equation*}
 \mathscr R_F(q)=Rq+O(r^m),
 \qquad
 D\mathscr R_F(q)=R+O(r^{m-1})
 \quad\text{as }r\downarrow0,
\end{equation*}
uniformly for $q=(x,y)$ with
$y\geq0$ and $c_1r\leq\abs q\leq c_2r$.
\end{proof}

\section{The index argument on the oriented double}

In this section, we use the two null line fields of $\mathring h_F$ and
the oriented double of the parameter disk to prove that
$\mathring h_F\equiv0$.

\begin{lemm}\label{lem:model-null-index}
Let $m\geq3$, $c\neq0$, and $P_m=c\Rea(z^m)$. For every $c_0\neq0$,
the symmetric matrix $c_0D^2P_m(q)$ has, at every $q\neq0$, the two
eigenvalues
\begin{equation*}
 \pm\abs{c_0c}m(m-1)\abs q^{m-2}.
\end{equation*}
Thus the associated symmetric bilinear form is nondegenerate and has
signature $(1,1)$. For each $q\neq0$, the set of nonzero vectors $v$
satisfying
$c_0D^2P_m(q)(v,v)=0$ is the union of exactly two distinct
one-dimensional subspaces of $\mathbb R^2$. Each of the resulting two
continuous line fields on $\mathbb R^2\setminus\{0\}$ has index
\begin{equation}\label{eq:negative-index}
 \ind_0=-\frac{m-2}{2}<0.
\end{equation}
\end{lemm}

\begin{proof}
Write $z=\varrho e^{i\theta}$ and set
$C_m=cm(m-1)$ and $\alpha=(m-2)\theta$. Direct differentiation gives
\begin{equation}\label{eq:hessian-matrix-model}
 D^2P_m
 =C_m\varrho^{m-2}
 \begin{pmatrix}
  \cos\alpha&-\sin\alpha\\
  -\sin\alpha&-\cos\alpha
 \end{pmatrix}.
\end{equation}
The square of the matrix in \eqref{eq:hessian-matrix-model} is the
identity. Hence the eigenvalues of $c_0D^2P_m$ are
\begin{equation}\label{eq:model-Hessian-eigenvalues}
 \pm\abs{c_0c}m(m-1)\varrho^{m-2}.
\end{equation}
They are nonzero and have opposite signs for every $\varrho>0$. Hence
$c_0D^2P_m$ is nondegenerate and has signature $(1,1)$ on
$\mathbb R^2\setminus\{0\}$.

For $v_\varphi=(\cos\varphi,\sin\varphi)$,
\begin{equation*}
 D^2P_m(v_\varphi,v_\varphi)
 =cm(m-1)\varrho^{m-2}
   \cos\bigl((m-2)\theta+2\varphi\bigr).
\end{equation*}
The two null line fields may be represented by the continuous angle
functions
\begin{equation*}
 \varphi_k(\theta)
 =-\frac{m-2}{2}\theta+\frac\pi4+\frac{k\pi}{2},
 \qquad k=0,1,
\end{equation*}
where angles are understood modulo $\pi$. Thus
\begin{equation}\label{eq:model-null-angle-change}
 \varphi_k(2\pi)-\varphi_k(0)=-(m-2)\pi.
\end{equation}
Dividing by $2\pi$ proves \eqref{eq:negative-index}.
\end{proof}

\begin{prop}\label{prop:interior-index}
Let $p\in\mathbb D$ be an interior A-umbilic, and let
$\WulffR{r_0}{c_p}$ be its comparison Wulff shape. Exactly one of the
following alternatives holds:
\begin{enumerate}[label=\textup{(\roman*)}]
 \item there is a neighborhood $U$ of $p$ such that
 $X(U)\subset\WulffR{r_0}{c_p}$;
 \item $p$ is an isolated A-umbilic, and there is an integer $m\geq3$
 such that each of the two null line fields has index
 \begin{equation*}
  \ind_p=-\frac{m-2}{2}<0.
 \end{equation*}
\end{enumerate}
\end{prop}

\begin{proof}
Take an interior A-umbilic $p$ as the origin and write the surface and
its comparison Wulff shape as graphs $u$ and $\widetilde u$ over their
common tangent plane. Lemma~\ref{lem:interior-graph-difference} gives
\eqref{eq:interior-linear-difference} and
\begin{equation*}
 w(0)=0,\qquad Dw(0)=0,\qquad D^2w(0)=0.
\end{equation*}

If $w\equiv0$ locally, alternative~\textup{(i)} holds. Suppose that
$w\not\equiv0$ in every neighborhood of the origin.
Lemma~\ref{lem:interior-expansion} then provides an orientation-preserving
linear coordinate system, an integer $m\geq3$, and a nonzero homogeneous
harmonic polynomial $P_m$ such that
\begin{equation*}
 D^2w=D^2P_m+o(\varrho^{m-2}),
 \qquad \varrho=\abs q.
\end{equation*}
The coordinate change preserves orientation and
hence the index. Lemma~\ref{lem:tensor-transfer} gives the corresponding
expansion for $\mathring h_F$:
\begin{equation*}
 \mathring h_F=c_0D^2P_m+R,
 \qquad c_0\neq0,
 \qquad \abs R=o(\varrho^{m-2}).
\end{equation*}

An orientation-preserving rotation gives
\[
 P_m=c\Rea(z^m)=c\varrho^m\cos(m\theta),
 \qquad c\neq0.
\]
By Lemma~\ref{lem:model-null-index}, the eigenvalues of the leading matrix are
$\pm\abs{c_0c}m(m-1)\varrho^{m-2}$.

Since $R=o(\varrho^{m-2})$, there is $\varrho_0>0$ such that, for
$0<\varrho\leq\varrho_0$,
\begin{equation}\label{eq:tracefree-hF-error-small-for-index}
 \abs R\leq\frac12\abs{c_0c}m(m-1)\varrho^{m-2}.
\end{equation}
By \eqref{eq:model-Hessian-eigenvalues} and
\eqref{eq:tracefree-hF-error-small-for-index}, the error is smaller than
half the absolute value of each eigenvalue of the leading matrix. Hence
the two eigenvalues of $\mathring h_F$ remain nonzero and have opposite
signs. Thus $\mathring h_F$ is nondegenerate and has signature $(1,1)$
on $0<\varrho\leq\varrho_0$, and $p$ is isolated.

Consider the homotopy
\begin{equation*}
 \mathring h_{F,t}:=c_0D^2P_m+tR,
 \qquad 0\leq t\leq1.
\end{equation*}
By \eqref{eq:tracefree-hF-error-small-for-index}, every
$\mathring h_{F,t}$ is nondegenerate and has signature $(1,1)$ on
$\{\varrho=r\}$ for
$0<r\leq\varrho_0$. For each $t$, it therefore has two distinct null
lines. Since these lines remain distinct, they can be followed
continuously as $t$ varies. At $t=0$, they are the two null line fields
of $c_0D^2P_m$; at $t=1$, they are $\cL_+$ and $\cL_-$.
Lemma~\ref{lem:model-null-index} and homotopy invariance
therefore give
\begin{equation*}
 \ind_p(\cL_+)=\ind_p(\cL_-)
 =-\frac{m-2}{2}<0.
\end{equation*}
\end{proof}

At an isolated boundary A-umbilic $p$, we join $\cL_+$ on a small
half-disk to $\cL_-$ on its reflected copy, using the $g_F$-orthogonal
boundary direction for the reflection. Proposition~\ref{prop:joachimsthal}
shows that the two lines agree along the boundary. We denote the index
of this local line field by $\ind_{\mathrm{dbl}}(p)$.

\begin{prop}
\label{prop:boundary-index}
Let $p\in\partial\mathbb D$ be a boundary A-umbilic, and let
$\WulffR{r_0}{c_p}$ be its comparison Wulff shape. Exactly one of the
following alternatives holds:
\begin{enumerate}[label=\textup{(\roman*)}]
 \item there is a neighborhood $U$ of $p$ relative to the closed disk
 such that $X(U)\subset\WulffR{r_0}{c_p}$;
 \item $p$ is an isolated A-umbilic, and there is an integer $m\geq3$ such
 that its doubled index is
 \begin{equation*}
  \ind_{\mathrm{dbl}}(p)=-\frac{m-2}{2}<0.
 \end{equation*}
\end{enumerate}
\end{prop}

\begin{proof}
Let $(x,y)$ be the normalized boundary coordinates of
Lemma~\ref{lem:analytic-geometric-compatibility}, with $p=(0,0)$,
$y\geq0$ on the disk, and $\varrho=\sqrt{x^2+y^2}$. Let $w$ be the
graph difference in Lemma~\ref{lem:boundary-expansion}. If
$w$ vanishes in a neighborhood of the origin relative to the closed
half-disk, alternative~\textup{(i)} holds. Suppose that it does not. Then
Lemma~\ref{lem:boundary-expansion} and Lemma~\ref{lem:tensor-transfer}
give
\begin{equation}\label{eq:boundary-index-tracefree-hF-expansion}
 \mathring h_F=c_0D^2P_m+\mathcal E,
 \qquad
 P_m=c\Rea(z^m),\qquad
 \abs{\mathcal E}=o(\varrho^{m-2}),\qquad c_0c\neq0.
\end{equation}
By Lemma~\ref{lem:model-null-index}, the two eigenvalues of the leading matrix
are $\pm\abs{c_0c}m(m-1)\varrho^{m-2}$.
Since $\abs{\mathcal E(x,y)}=o(\varrho^{m-2})$ as $(x,y)\to(0,0)$
with $y\geq0$, there is $\delta>0$ such that
\[
 \abs{\mathcal E(x,y)}
 \leq \frac12\abs{c_0c}m(m-1)\varrho^{m-2}
 \qquad (y\geq0,\ 0<\varrho<\delta).
\]
Thus the two eigenvalues of $\mathring h_F$ have opposite signs at every
such point. In particular, its two null lines are defined there, and
$p$ is the only A-umbilic in this neighborhood.

To compute the index of the geometric double, set
\begin{equation*}
 S^1_+:=\{q=(q_1,q_2):\abs q=1,\ q_2\geq0\},
 \qquad
 S^1_-:=\{q=(q_1,q_2):\abs q=1,\ q_2\leq0\}.
\end{equation*}
For $q=(q_1,q_2)\in S^1_-$, write $q^*:=(q_1,-q_2)\in S^1_+$.
For small $r>0$, define a loop in the doubled coordinate neighborhood by
\begin{equation*}
 \Gamma_r(q):=
 \begin{cases}
  rq,&q\in S^1_+,\\
  \mathscr R_F(rq^*),&q\in S^1_-.
 \end{cases}
\end{equation*}
The doubled line field along this loop is
\begin{equation}\label{eq:boundary-doubled-line-field}
 (\widehat\cL_r)_{\Gamma_r(q)}:=
 \begin{cases}
  \cL_+(rq),&q\in S^1_+,\\
  D\mathscr R_F(rq^*)\cL_-(rq^*),&q\in S^1_-.
 \end{cases}
\end{equation}
The reflection $\mathscr R_F$ fixes the boundary line and reverses the
$g_F$-conormal direction. The two definitions of $\Gamma_r$ therefore
coincide at $q=(\pm1,0)$. Proposition~\ref{prop:joachimsthal} shows that
the two definitions of $\widehat\cL_r$ also coincide at these points.
Hence $\widehat\cL_r$ is a continuous line field along $\Gamma_r$.
The coordinate basis identifies every tangent space in this chart with
$\mathbb R^2$; with this identification,
\eqref{eq:boundary-doubled-line-field} defines a continuous map
$S^1\to\mathbb{RP}^1$.

Equation~\eqref{eq:boundary-index-tracefree-hF-expansion} gives,
uniformly for $q\in S^1_+$,
\begin{equation*}
 r^{2-m}\mathring h_F(rq)\longrightarrow c_0D^2P_m(q).
\end{equation*}
Let $q_0=(1,0)$. As $r\to0$, the line $\cL_+(rq_0)$ converges to one of
the two null lines of $D^2P_m(q_0)$. Denote this limiting line by
$\cM(q_0)$. Since the two null lines of $D^2P_m$ are distinct at every
point of $S^1$, $\cM(q_0)$ extends uniquely to a continuous null line
field $\cM$ on $S^1$. Since the model eigenvalues are uniformly separated
from zero,
\begin{equation*}
 \cL_+(rq)\longrightarrow\cM(q)
 \qquad\text{uniformly for }q\in S^1_+.
\end{equation*}

For $q\in S^1_-$, the line $\cL_-(rq^*)$ converges to the model null
line different from $\cM(q^*)$. Since $P_m$ is even in $q_2$,
\begin{equation*}
 D^2P_m(q)=R D^2P_m(q^*)R,
 \qquad R=\operatorname{diag}(1,-1).
\end{equation*}
At $q_0$, the matrix $D^2P_m(q_0)$ is diagonal and $R$ interchanges its
two null lines. Continuity along $S^1_-$ therefore shows that $R$ sends
the limiting line of $\cL_-(rq^*)$ to $\cM(q)$. Lemma~\ref{lem:analytic-geometric-compatibility}
now gives
\begin{equation*}
 D\mathscr R_F(rq^*)\cL_-(rq^*)\longrightarrow\cM(q)
 \qquad\text{uniformly for }q\in S^1_-.
\end{equation*}
Thus $\widehat\cL_r\to\cM$ uniformly on $S^1$. For small $r$,
each line $\widehat\cL_r(q)$ can be rotated continuously through
the smaller angle to $\cM(q)$. Hence the two line fields have the
same index. Equation~\eqref{eq:model-null-angle-change} therefore gives
\[
 \ind_{\mathrm{dbl}}(p)
 =\ind_0(\cM)
 =-\frac{m-2}{2}.
\]
\end{proof}

\begin{lemm}
\label{lem:global-propagation}
Let $p\in\Dbar$, and let $U$ be a neighborhood of $p$ in $\Dbar$.
If $X(U)\subset\WulffR{r_0}{c_p}$, then
$X(\Dbar)\subset\WulffR{r_0}{c_p}$.
Consequently, $\mathring h_F\equiv0$.
\end{lemm}

\begin{proof}
The radius $r_0=2/H_F$ is fixed, and the local agreement determines a
unique center $c$. Let $\mathcal O$ be the set of all $q\in\Dbar$ for
which there is a neighborhood $U_q$ such that
$X(U_q)\subset\WulffR{r_0}{c}$ and the orientation induced by $X$ agrees
with the orientation of $Y(\xi)=c+r_0\Phi(\xi)$. By assumption,
$\mathcal O$ is nonempty, and it is open. We prove that it is also
closed.

Suppose $p_j\in\mathcal O$ and $p_j\to p$. On $\mathcal O$ one has
\begin{equation}\label{eq:fixed-model-center-identity}
 X-r_0\nu_F=c,
 \qquad d\nu_F=r_0^{-1}dX.
\end{equation}
Passing to the limit shows that $X(p)\in\WulffR{r_0}{c}$, that the oriented
tangent planes agree at $p$, and that the surface and the fixed model
have the same second fundamental form there. Thus, in common graph
coordinates centered at $p$, their difference satisfies
$w(0)=0$, $Dw(0)=0$, and $D^2w(0)=0$.

If $p\in\mathbb D$, Lemma~\ref{lem:interior-graph-difference} gives
the uniformly elliptic equation \eqref{eq:interior-linear-difference}.
Suppose $p\in\partial\mathbb D$. Since
$X(p)\in E_3^\perp$ and
$\inner{\nu_F(p)}{-E_3}=\omega_0$, the $E_3$-component of
\eqref{eq:fixed-model-center-identity} gives
\begin{equation*}
 \inner{c}{E_3}
 =\inner{X(p)-r_0\nu_F(p)}{E_3}=r_0\omega_0.
\end{equation*}
For a point $Y(\xi)=c+r_0\Phi(\xi)$ on the fixed Wulff shape,
\begin{equation*}
 \inner{Y(\xi)}{E_3}
 =r_0\bigl(\omega_0+\inner{\Phi(\xi)}{E_3}\bigr).
\end{equation*}
Consequently,
\begin{equation*}
 Y(\xi)\in\partial\mathbb R^3_+
 \quad\Longleftrightarrow\quad
 \inner{\Phi(\xi)}{-E_3}=\omega_0.
\end{equation*}
Hence the fixed Wulff shape satisfies the Young condition along its
intersection with the support plane, and the graph difference satisfies
the conormal boundary condition
\eqref{eq:linear-conormal}.

For large $j$, the common graph neighborhood contains points
$q_j\to0$, $q_j\neq0$, corresponding to $p_j$, and $w$ vanishes on a
neighborhood of every $q_j$. If $w$ did not vanish near the origin, the
corresponding local expansion in Lemma~\ref{lem:interior-expansion} or
Lemma~\ref{lem:boundary-expansion} would give
\[
 D^2w(q)=D^2P_m(q)+E(q),
 \qquad P_m=c_m\Rea(z^m),\qquad c_m\neq0,\quad m\geq3,
\]
where
\begin{equation*}
 \frac{\abs{E(q)}}{\abs q^{m-2}}\longrightarrow0
 \qquad\text{as }q\longrightarrow0.
\end{equation*}
But the two eigenvalues of $D^2P_m(q)$ have absolute value
$\abs{c_m}m(m-1)\abs q^{m-2}$ for every $q\neq0$. Hence the expansion
implies, for all sufficiently large $j$,
\[
 \abs{D^2w(q_j)}
 \geq\frac12\abs{c_m}m(m-1)\abs{q_j}^{m-2}>0,
\]
contradicting the fact that $w$ vanishes near $q_j$. Therefore the
expansion alternative in Lemma~\ref{lem:interior-expansion} or
Lemma~\ref{lem:boundary-expansion} is impossible, and $w$ vanishes near
$p$. Thus $p\in\mathcal O$, proving that $\mathcal O$ is closed.

Connectedness of the disk gives $\mathcal O=\Dbar$. The oriented normal
then agrees everywhere with the model normal, so
$d\nu_F=r_0^{-1}dX$ and hence $\mathring h_F\equiv0$.
\end{proof}

Suppose $\mathring h_F\not\equiv0$.
Lemma~\ref{lem:global-propagation} rules out local agreement with a Wulff
surface. Propositions~\ref{prop:interior-index} and
\ref{prop:boundary-index} then show that every point of $\cU$ is isolated
and gives a negative index contribution on the oriented double.
Since $\cU$ is closed in the compact disk, it is finite.
Take two abstract copies $\Dbar_+$ and $\Dbar_-$ of the parameter disk
and identify corresponding boundary points. We orient
$\Dbar_+$ by the original orientation and $\Dbar_-$ by the opposite
orientation. Define
\begin{equation}\label{eq:oriented-double}
 \widehat\Sigma:=\Dbar_+\cup_{\partial\mathbb D}\Dbar_-
\end{equation}
and
\begin{equation*}
 \widehat\cU=
 \{p_+,p_-:p\in\cU\cap\mathbb D\}
 \cup(\cU\cap\partial\mathbb D).
\end{equation*}
The two copies have opposite orientations along their common boundary, so \(\widehat\Sigma\) is an oriented sphere; moreover, \(\widehat{\mathcal U}\) is finite.

\begin{defi}\label{def:Hopf-line-index}
Let $L$ be an unoriented line field with an isolated singularity $p$
on an oriented surface $M$. Choose a positively oriented small circle
$\gamma:[0,2\pi]\to M$ around $p$ and an oriented tangent frame
$(e_1,e_2)$ on a disk containing it. Write
\[
 L_{\gamma(t)}
 =\Span\{\cos\theta(t)e_1+\sin\theta(t)e_2\},
\]
where $\theta:[0,2\pi]\to\mathbb R$ is continuous. Define
\begin{equation*}
 \ind_p(L)
 :=\frac{\theta(2\pi)-\theta(0)}{2\pi}
 \in\frac12\mathbb Z.
\end{equation*}
The index is independent of these choices and is unchanged if the
orientation of $M$ is reversed.
\end{defi}

\begin{prop}
\label{prop:line-field-PH}
With the half-integer normalization of
Definition~\ref{def:Hopf-line-index}, if $M$ is a closed oriented surface
and $L$ is a continuous line field with only finitely many isolated
singularities, then
\begin{equation}\label{eq:line-field-PH-general}
 \sum_{q\in\Sing(L)}\ind_q(L)=\chi(M).
\end{equation}
\end{prop}

\begin{proof}
This is the Poincar\'e--Hopf theorem for line fields with the
normalization in Definition~\ref{def:Hopf-line-index}; see
\cite[p.~113]{Hopf1983} .
\end{proof}

\begin{lemm}
\label{lem:doubled-null-line-field}
Use $\cL_+$ on $\Dbar_+$ and $\cL_-$ on $\Dbar_-$ to define
$\widehat\cL$. Then $\widehat\cL$ is a continuous unoriented line field on
$\widehat\Sigma\setminus\widehat\cU$.
\end{lemm}

\begin{proof}
By Lemma~\ref{lem:global-null-label}, the two line fields are globally
defined. Fix a non-A-umbilic boundary
point and choose boundary coordinates
$(s,t)$ on the original disk, where $t\geq0$, such that along $t=0$ one
has $\partial_s=\tau$ and $\partial_t$ is a positive multiple of the
anisotropic conormal $\eta_F$. Proposition~\ref{prop:joachimsthal}
gives
\[
 \mathring h_F(\partial_s,\partial_t)=0\qquad\text{on }t=0.
\]
Shrink the coordinate neighborhood so that $\mathring h_F$ is nonzero
on the boundary arc. Since it is trace-free, it has signature $(1,1)$
there. Along $t=0$, write its two diagonal coefficients as
\[
 \mathring h_{F,11}:=\mathring h_F(\partial_s,\partial_s),\qquad
\mathring h_{F,22}:=\mathring h_F(\partial_t,\partial_t).
\]
Then $\mathring h_{F,11}\mathring h_{F,22}<0$. Therefore
\begin{equation*}
 a:=\sqrt{-\frac{\mathring h_{F,11}}{\mathring h_{F,22}}}>0,
 \qquad
 \{[v]:\mathring h_F(v,v)=0\}
 =\Span\{\partial_s+a\partial_t\}
  \cup\Span\{\partial_s-a\partial_t\}.
\end{equation*}
For some $\varepsilon\in\{1,-1\}$, $\cL_+$ is spanned by
$\partial_s+\varepsilon a\partial_t$ and $\cL_-$ is spanned by
$\partial_s-\varepsilon a\partial_t$ on each boundary arc avoiding
$\cU$.

On the second copy, use the coordinate $y=-t$.
Since $\partial_t=-\partial_y$ there, the line selected from $\cL_-$
becomes
\begin{equation*}
 \Span\{\partial_s-\varepsilon a\partial_t\}
 =\Span\{\partial_s+\varepsilon a\partial_y\}.
\end{equation*}
This agrees with the line from $\cL_+$ on the first copy. Hence the two
fields combine to define a continuous line field $\widehat\cL$ on
$\widehat\Sigma\setminus\widehat\cU$.
\end{proof}

\begin{rema}
The construction doubles the parameter disk, not the immersed surface;
no reflection symmetry of $F$ or smooth extension of $\mathring h_F$ is
required. At a boundary A-umbilic,
Lemma~\ref{lem:analytic-geometric-compatibility} and
Proposition~\ref{prop:boundary-index} identify the local index of this
geometric double.
\end{rema}

\begin{lemm}
\label{lem:doubled-index-sum}
For each $p\in\cU$, let $m_p\geq3$ be the integer given by
Proposition~\ref{prop:interior-index} or
Proposition~\ref{prop:boundary-index}, according as $p$ is an interior
or boundary point. The singularities of $\widehat\cL$ are exactly the
points of $\widehat\cU$, and
\begin{equation}\label{eq:doubled-index-sum}
 \sum_{q\in\Sing(\widehat\cL)}\ind_q(\widehat\cL)
 =2\sum_{p\in\cU\cap\mathbb D}
       \left(-\frac{m_p-2}{2}\right)
  +\sum_{p\in\cU\cap\partial\mathbb D}
       \left(-\frac{m_p-2}{2}\right).
\end{equation}
The sum is strictly negative when
$\cU\neq\varnothing$ and is zero when $\cU=\varnothing$.
\end{lemm}

\begin{proof}
The singular set is $\widehat\cU$, since $\mathring h_F$ has two distinct
null lines
wherever $\mathring h_F\neq0$.

For $p\in\cU\cap\mathbb D$,
Proposition~\ref{prop:interior-index} supplies an integer $m_p\geq3$ and gives
\begin{equation*}
 \ind_{p_+}(\widehat\cL)
 =\ind_{p_-}(\widehat\cL)
 =-\frac{m_p-2}{2}<0.
\end{equation*}
In a $g_F$-orthonormal frame, $\cL_-$ is obtained from $\cL_+$ by a
rotation through $\pi/2$, which does not change the index. Reversing the
orientation on the second copy also leaves the index unchanged by
Definition~\ref{def:Hopf-line-index}. Therefore the two copies give the
same index.

For $p\in\cU\cap\partial\mathbb D$,
Proposition~\ref{prop:boundary-index} supplies an integer $m_p\geq3$ and
gives the doubled index
\begin{equation*}
 \ind_p(\widehat\cL)=\ind_{\mathrm{dbl}}(p)
 =-\frac{m_p-2}{2}<0.
\end{equation*}

Adding these contributions gives \eqref{eq:doubled-index-sum}, and its
sign follows from $m_p\geq3$.
\end{proof}

\begin{proof}[Proof of Theorem~\ref{thm:main}]
Suppose $\mathring h_F\not\equiv0$.

Apply \eqref{eq:line-field-PH-general} to the line field
$\widehat\cL$ on the oriented double \eqref{eq:oriented-double}. Since
$\widehat\Sigma\simeq\bS^2$, it gives
\begin{equation*}
 \sum_{q\in\Sing(\widehat\cL)}\ind_q(\widehat\cL)
 =\chi(\widehat\Sigma)=\chi(\bS^2)=2.
\end{equation*}
If $\cU=\varnothing$, the index sum is zero; otherwise, it is negative
by \eqref{eq:doubled-index-sum}. Both contradict the value $2$. Hence
\begin{equation*}
 \mathring h_F\equiv0.
\end{equation*}
The boundary condition and transversality show that its image is the
corresponding Wulff cap in $\overline{\mathbb R^3_+}$. This proves
Theorem~\ref{thm:main}.
\end{proof}

\

\appendix

\section{Proof of the interior expansion}
\label{app:interior-expansion}

\begin{proof}[Proof of Lemma~\ref{lem:interior-expansion}]
Without loss of generality, $A(0)=I$. Writing
$b^j=\partial_i a^{ij}$, the equation becomes
\[
 a^{ij}w_{ij}+b^jw_j=0,
 \qquad b^j\in C^{0,\alpha}.
\]
If $w$ is locally constant, then $w\equiv0$ near the origin. Otherwise,
the Hartman--Wintner theorem
\cite[Theorem~(H.--W.) and Remark~1.1, p.~231]{Alessandrini1987}
gives an integer $m\geq1$ and a nonzero homogeneous polynomial $P_m$
satisfying
\[
 a^{ij}(0)(P_m)_{ij}=0
\]
such that
\[
 w=P_m+o(\abs q^m),
 \qquad Dw=DP_m+o(\abs q^{m-1}),
\]
where the little-$o$ terms are uniform for $q/\abs q\in\bS^1$. Since
$A(0)=I$, one has $\Delta P_m=0$. The three vanishing conditions imply
$m\geq3$.

Choose fixed closed annuli
\[
 K\Subset K'\Subset\{1/2<\abs q<2\}
\]
with $\bS^1\subset\operatorname{int}K$, and set
\[
 w_\varepsilon(q)=\varepsilon^{-m}w(\varepsilon q),
 A_\varepsilon(q)=A(\varepsilon q),
 \qquad
 V_\varepsilon=w_\varepsilon-P_m.
\]
Then $V_\varepsilon\to0$ in $C^1(K')$ and
\[
 A_\varepsilon^{ij}(V_\varepsilon)_{ij}
 +\varepsilon b^j(\varepsilon q)(V_\varepsilon)_j
 =f_\varepsilon,
\]
where
\[
 f_\varepsilon
 =-(A_\varepsilon^{ij}-\delta^{ij})(P_m)_{ij}
  -\varepsilon b^j(\varepsilon q)(P_m)_j.
\]
Since $A\in C^{1,\alpha}$ and $A(0)=I$,
\[
 A_\varepsilon\longrightarrow I\quad\text{in }C^{1,\alpha}(K'),
 \qquad
 \|\varepsilon b(\varepsilon\,\cdot)\|_{C^{0,\alpha}(K')}
 +\|f_\varepsilon\|_{C^{0,\alpha}(K')}=O(\varepsilon).
\]
The interior Schauder estimate
\cite[Theorem~6.2]{GilbargTrudinger2001} gives
\[
 \|V_\varepsilon\|_{C^{2,\alpha}(K)}
 \leq C\left(
 \|V_\varepsilon\|_{C^0(K')}
 +\|f_\varepsilon\|_{C^{0,\alpha}(K')}
 \right)\longrightarrow0.
\]
By homogeneity,
\[
 D^2V_\varepsilon(q)
 =\varepsilon^{2-m}
 \bigl(D^2w(\varepsilon q)-D^2P_m(\varepsilon q)\bigr).
\]
Taking $\varepsilon=\abs x$ and $q=x/\abs x$ proves
\eqref{eq:int-exp-2}, uniformly for $q\in\bS^1$.
\end{proof}

\section{Proof of the boundary expansion}
\label{app:boundary-expansion}

\begin{proof}[Proof of Lemma~\ref{lem:boundary-expansion}]
Along $t=0$, set
\[
 \beta(x)=\frac{a^{12}(x,0)}{a^{22}(x,0)},
 \qquad \Psi(x,y)=(x+\beta(x)y,y),
 \qquad v=w\circ\Psi.
\]
Restrict the half-disk so that $\Psi$ is an orientation-preserving
$C^{2,\alpha}$ diffeomorphism preserving $y=0$ and $y\geq0$. In the
$(x,y)$-coordinates, the coefficient matrix is
\[
 B=\abs{\det D\Psi}\,D\Psi^{-1}(A\circ\Psi)D\Psi^{-T}.
\]
Direct multiplication on $y=0$ gives $B^{12}(x,0)=0$.
The conormal condition becomes $e_y\cdot B\nabla v=0$; hence
$v_y(x,0)=0$. The coordinate change preserves
$v(0)=0$, $Dv(0)=0$, and $D^2v(0)=0$.

At the origin, $B(0)$ is diagonal and positive definite. Choose a fixed
positive diagonal linear map $L$, preserving $y=0$ and $y\geq0$, so
that the coefficient matrix in the $\Psi\circ L$-coordinates is a
positive multiple of the identity at the origin. Dividing the matrix by
this positive constant changes neither the divergence equation nor the
homogeneous conormal boundary condition. We may therefore assume
\begin{equation}\label{eq:normalized-boundary-coordinates}
 B(0)=I,\qquad B^{12}(x,0)=0,\qquad v_y(x,0)=0.
\end{equation}

Let $R=\operatorname{diag}(1,-1)$. For $y<0$, define
\[
 v^e(x,y)=v(x,-y),
 \qquad \widehat B(x,y)=RB(x,-y)R,
\]
For $y\geq0$, set $v^e=v$ and $\widehat B=B$. Thus the diagonal
coefficients are extended evenly and the mixed coefficient oddly. Since
$B^{12}(x,0)=0$, the matrix $\widehat B$ is Lipschitz and uniformly
elliptic. Moreover, $v_y=0$ on $y=0$, and tangential differentiation
gives $v_{xy}=0$ . Hence the even extension belongs to
$C^{2,\alpha}$ across the boundary line $y=0$. Splitting a test function
into its upper and lower half-disk parts and using the substitution
$y\mapsto-y$ in the lower part gives
\[
 \Div(\widehat B\nabla v^e)=0
\]
in the weak sense. Equivalently,
\begin{equation*}
 \widehat B^{ij}(v^e)_{ij}
 +\widehat b^j(v^e)_j=0 ,
 \qquad
 \widehat b^j:=\partial_i\widehat B^{ij}\in L^\infty.
\end{equation*}

If $v^e$ is locally constant, then $v^e\equiv v^e(0)=0$, which gives
alternative~\textup{(a)}. Otherwise, the Lipschitz-coefficient
Hartman--Wintner theorem
\cite[Theorem~(H.--W.) and Remark~1.1, p.~231]{Alessandrini1987}
applies. In the normalized coordinates
\eqref{eq:normalized-boundary-coordinates}, there exist an integer
$m\geq1$ and a nonzero homogeneous harmonic polynomial $P_m$ such that
\[
 v^e=P_m+o(\varrho^m),
 \qquad Dv^e=DP_m+o(\varrho^{m-1}).
\]
The evenness of $v^e$ implies that $P_m$ is even in $y$. Hence
\[
 P_m=c\Rea(z^m),\qquad c\neq0.
\]
The vanishing of $v$, $Dv$, and $D^2v$ at the origin implies $m\geq3$.

In the relative topology of $\{y\geq0\}$, choose compact half-annuli
\begin{equation*}
 K\Subset K'\Subset
 \{q=(x,y):1/2<\abs q<2,\ y\geq0\},
\end{equation*}
with the closed upper unit semicircle contained in the relative interior
of $K$. Denote
\[
 v_r(q)=r^{-m}v(rq),
 \qquad B_r(q)=B(rq),
 \qquad V_r=v_r-P_m.
\]
Writing $b^j=\partial_iB^{ij}$, the function $V_r$ satisfies
\[
 B_r^{ij}(V_r)_{ij}+r b^j(rq)(V_r)_j=f_r,
\]
where
\begin{equation*}
 f_r=-(B_r^{ij}-\delta^{ij})(P_m)_{ij}
     -r b^j(rq)(P_m)_j.
\end{equation*}
Since $B\in C^{1,\alpha}$ and $B(0)=I$,
\begin{equation}\label{eq:boundary-scaled-coefficient-bounds}
 \|B_r-I\|_{C^{1,\alpha}(K')}
 +\|r b(r\,\cdot)\|_{C^{0,\alpha}(K')}
 +\|f_r\|_{C^{0,\alpha}(K')}
 \leq Cr.
\end{equation}
On $y=0$, one has $B_r^{12}=0$ and $(P_m)_y=0$. The boundary condition
for $V_r$ is therefore homogeneous:
\begin{equation*}
 B_r^{2j}(V_r)_j=0
 \quad\text{on }K'\cap\{y=0\}.
\end{equation*}
Moreover, $B_r^{22}\geq\lambda>0$, and the $C^{1,\alpha}$ norm of the
boundary vector $(B_r^{21},B_r^{22})$ is bounded independently of $r$.

Cover $K$ by finitely many sets $U_\ell\Subset U'_\ell$, with each
$U'_\ell$ contained in the relative interior of $K'$. Use interior
disks in $\{y>0\}$ and flat half-disks whose flat boundary lies in
$\{y=0\}$. On the interior disks, we apply the interior Schauder estimate
\cite[Theorem~6.2]{GilbargTrudinger2001}. On each flat half-disk,
the boundary condition is homogeneous Neumann, since
$B_r^{12}=0$ and $B_r^{22}>0$ on $y=0$. Hence the boundary
Schauder estimate
\cite[Theorem~6.30]{GilbargTrudinger2001} applies.
Combining these estimates over the finite cover, we obtain
\begin{equation*}
 \|V_r\|_{C^{2,\alpha}(K)}
 \leq C\left(
  \|V_r\|_{C^0(K')}
  +\|f_r\|_{C^{0,\alpha}(K')}
 \right).
\end{equation*}
The Hartman--Wintner expansion gives $V_r\to0$ in $C^1(K')$.
Consequently, by \eqref{eq:boundary-scaled-coefficient-bounds},
\[
 \|V_r\|_{C^{2,\alpha}(K)}\longrightarrow0.
\]
Write \(q=r\omega\), where \(r=|q|>0\) and \(\omega\) lies on the closed upper unit semicircle. Then
\[
D^2v(q)-D^2P_m(q)
=r^{m-2}D^2V_r(\omega)
=o(r^{m-2})
\]
The estimate is uniform in \(\omega\). This proves \eqref{eq:bd-exp-2}
and completes alternative \textup{(b)}.
\end{proof}

\section{Graph equations and ellipticity}
\label{app:graph}

Use the notation of Lemma~\ref{lem:interior-graph-difference}. Direct differentiation gives
\[
 X_1\times X_2
 =\nu(p)+u_1(\zeta\times e_2)+u_2(e_1\times\zeta)
 =V_\zeta(Du),
 \qquad
 \inner{V_\zeta(q)}{\zeta}=\inner{\nu(p)}{\zeta}.
\]
For $\lambda=(\lambda_1,\lambda_2)$, set
\[
 \xi_\lambda
 :=\lambda_1(\zeta\times e_2)+\lambda_2(e_1\times\zeta),
 \qquad
 \nu(q):=\frac{V_\zeta(q)}{\abs{V_\zeta(q)}}.
\]
If $\xi_\lambda^\top$ denotes the Euclidean projection onto
$\nu(q)^\perp$, one-homogeneity gives
\[
 D^2G_\zeta(q)[\lambda,\lambda]
 =\frac1{\abs{V_\zeta(q)}}
  \inner{A_F(\nu(q))\xi_\lambda^\top}{\xi_\lambda^\top}.
\]
Here $\xi_\lambda\perp\zeta$. If $\xi_\lambda^\top=0$, then
$\xi_\lambda$ is parallel to $V_\zeta(q)$. Taking the inner product
with $\zeta$ gives $\xi_\lambda=0$, and hence $\lambda=0$. It follows
from the strict ellipticity of $A_F$ that $D^2G_\zeta(q)$ is positive
definite. Compactness then gives uniform ellipticity on every bounded
subset of $\mathbb R^2$.

The anisotropic area of the graph is
$\int G_\zeta(Du)\dd x_1\dd x_2$. For
$u_\varepsilon=u+\varepsilon\varphi$, differentiation with respect to
$\varepsilon$ gives
\[
 -\int\Div DG_\zeta(Du)\,\varphi\dd x_1\dd x_2.
\]
For this variation,
\[
 \inner{\varphi\zeta}{\nu_u}\dd A
 =\varphi\inner{\zeta}{V_\zeta(Du)}\dd x_1\dd x_2
 =\inner{\zeta}{\nu(p)}\varphi\dd x_1\dd x_2.
\]
With the convention \eqref{eq:HF-def}, comparison of the two first
variations gives \eqref{eq:interior-graph-equations}.

For $W=\rho\xi\neq0$ and $\chi\in\mathbb R^3$, let $\chi^\top$ be the
projection of $\chi$ onto $\xi^\perp$. The differential of
$W\mapsto W/\abs W$ sends $\chi$ to $\chi^\top/\rho$. Differentiating
$\bar\nabla F(\rho\xi)=\Phi(\xi)$ gives
\[
 \bar\nabla^2F(W)[\chi,\chi]
 =\frac1\rho\inner{A_F(\xi)\chi^\top}{\chi^\top}.
\]
Choose $0<\lambda_F\leq\Lambda_F$ such that
\[
 \lambda_F\abs v^2
 \leq\inner{A_F(\xi)v}{v}
 \leq\Lambda_F\abs v^2
 \qquad(\xi\in\mathbb S^2,\ v\perp\xi).
\]

In the notation of \eqref{eq:adapted-graph}, for
$\zeta=(\zeta_1,\zeta_2)$ set
\[
 \xi=-b\zeta_1\tau-\zeta_2E_3,
 \qquad \rho(q)=\abs{V(q)},
 \qquad \nu(q)=\frac{V(q)}{\rho(q)}.
\]
The chain rule gives \eqref{eq:G-Hessian-exact}. Moreover,
\[
 \xi\perp\bar\nu,
 \qquad \inner{V(q)}{\bar\nu}=b,
 \qquad \rho(q)\geq b.
\]
Writing $V=b\bar\nu+V^\perp$, where
$V^\perp\perp\bar\nu$, gives
\[
 \abs{\xi^\top}^2
 =\abs\xi^2-\frac{\inner{\xi}{V}^2}{\rho^2}
 \geq\frac{b^2}{\rho^2}\abs\xi^2,
\]
while
\[
 \min\{b^2,1\}\abs\zeta^2
 \leq\abs\xi^2=b^2\zeta_1^2+\zeta_2^2
 \leq\max\{b^2,1\}\abs\zeta^2.
\]
Consequently, on $\abs q\leq M$ one may take
\[
 \lambda_M
 =\frac{\lambda_Fb^2\min\{b^2,1\}}
 {\bigl(\sup_{\abs q\leq M}\rho(q)\bigr)^3},
 \qquad
 \Lambda_M=\frac{\Lambda_F\max\{b^2,1\}}b,
\]
which proves \eqref{eq:G-uniform-ellipticity}.

Finally, by \eqref{eq:oblique-graph-normal},
\[
 V(Du)=X_s\times X_t=\nu(p)-bu_1\tau-u_2E_3,
 \qquad
 \nu_u=\frac{V(Du)}{\abs{V(Du)}}.
\]
One-homogeneity gives
\[
 F(\nu_u)\dd A=F(V(Du))\dd s\dd t=G(Du)\dd s\dd t.
\]
For the variation $u_\varepsilon=u+\varepsilon\varphi$, differentiation
with respect to $\varepsilon$ gives
\[
 \left.\frac{d}{d\varepsilon}\right|_{0}
 \int G(Du_\varepsilon)\dd s\dd t
 =-\int\Div(DG(Du))\varphi\dd s\dd t
  +\int_{\partial\Omega}DG(Du)\cdot n_\Omega\,\varphi\dd s.
\]
Since $\inner{V(Du)}{\bar\nu}=b$, the first variation formula for
anisotropic area gives
\[
 \int H_F\inner{\varphi\bar\nu}{\nu_u}\dd A
 =b\int H_F\varphi\dd s\dd t.
\]
Thus $\Div(DG(Du))=-bH_F$. Along $q_2=0$, the graph direction
$\bar\nu$ lies in $\partial\mathbb R^3_+$ and
\[
 G_{q_2}(Du)
 =-\bar\nabla F(V(Du))\cdot E_3
 =\inner{\nu_F}{-E_3}.
\]
Hence the Young condition is $G_{q_2}(Du)=\omega_0$.

\end{document}